\documentclass[a4paper,oneside,10pt]{amsart}

\usepackage{parskip}
\usepackage{a4wide}
\usepackage[T1]{fontenc} 
\usepackage[ansinew]{inputenc}
\usepackage{lmodern} 
\usepackage{graphicx}
\usepackage{amsmath}
\usepackage{amsthm}
\usepackage{amsfonts}
\usepackage{amssymb}
\usepackage{setspace}
\usepackage{mathrsfs}
\usepackage[all]{xy}
\usepackage{enumerate}
\usepackage{tikz}

\usetikzlibrary{arrows.meta,positioning}

\newtheorem{theorem}{Theorem}[section]

\newtheorem{corollary}{Corollary}

\newtheorem{lemma}[theorem]{Lemma}
\newtheorem{proposition}[theorem]{Proposition}

\theoremstyle{definition}
\newtheorem{definition}[theorem]{Definition}
\newtheorem{assumption}[theorem]{Assumption}
\newtheorem{example}{Example}

\newtheorem{remark}[theorem]{Remark}

\newtheorem*{theorem*}{Theorem}
\newtheorem*{proposition*}{Proposition}
\newtheorem*{definition*}{Definition}
\numberwithin{equation}{section}



\def\L{\mathrm{L}}
\def\LLL{\mathscr{L}}
\def\c{\mathrm{c}}
\def\W{\mathrm{W}}
\def\dom{\mathrm{D}}
\def\dd{\mathrm{d}}
\def\ee{\mathrm{e}}
\def\vv{\mathrm{v}}
\def\BC{\mathrm{C}_{\mathrm{b}}}
\def\RR{\mathbb{R}}
\def\CC{\mathbb{C}}
\def\NN{\mathbb{N}}
\def\BB{\mathbb{B}}
\def\QQ{\mathbb{Q}}
\def\Id{\mathrm{I}}
\def\semis{\mathscr{P}}
\def\diag{\mathrm{diag}}
\def\esssup{\mathop{\mathrm{ess}\,\sup}}
\def\tlim{\mathop{\tau-\mathrm{lim}}}

\newenvironment{abc}{\begin{enumerate}[{\rm (a)}]}{\end{enumerate}}

\newenvironment{iiv}{\begin{enumerate}[{\rm (i)}]}{\end{enumerate}}
\thanks{The authors are grateful to University of Wuppertal for the possibility of funding the stay of the first author at the University of Ljubljana within the Erasmus exchange program. {The first author was supported by the DAAD-TKA Project 308019 ``\emph{Coupled systems and innovative time integrators}''}. The second author acknowledges financial support from the Slovenian Research Agency, Grant No.~P1-0222. Both  authors are thankful to B\'{a}lint Farkas for encouraging the topic of this paper as well as for many beneficial discussions, and to Eszter Sikolya for her valuable comments.}

\subjclass{Primary: 35R02; Secondary: 35F46, 47D06, 46A70.}
 \keywords{transport equations, infinite metric graphs, bi-continuous operator semigroups, Bochner $\L^{\infty}$-spaces, perturbations.}

\begin{document}

\title{Bi-Continuous Semigroups for Flows on Infinite Networks} 
\author{Christian Budde}
\address{North-West University, School of Mathematical and Statistical Sciences, Private Bag X6001-209, Potchefstroom 2520, South Africa}
\email{christian.budde@nwu.ac.za}
\author{Marjeta Kramar Fijav\v{z}}
\address{University of Ljubljana, Faculty of Civil and Geodetic Engineering, Jamova 2, SI-1000 Ljubljana, Slovenia / Institute of Mathematics, Physics, and Mechanics, Jadranska 19, SI-1000 Ljubljana, Slovenia}
\email{marjeta.kramar@fgg.uni-lj.si}

\begin{abstract}
We study transport processes on infinite metric graphs with non-constant velocities and matrix boundary conditions in the $\L^{\infty}$-setting. We apply the theory of bi-continuous operator semigroups to obtain well-posedness of the problem under different assumptions on the velocities and for general stochastic matrices appearing in the boundary conditions.
\end{abstract}


\maketitle

\section{Introduction}

Consider a huge network of {possibly} unknown size {but known structure.
To model this situation, we may take an infinite graph and equip it with the appropriate combinatorial assumptions.}
Along the edges of the network some transport processes take place that are coupled in the vertices in which the edges meet. This means that we 
consider each edge as an interval and describe functions on it, that is, we consider a \emph{metric graph}. {Such} systems of partial differential equations on a metric graph are also known as \emph{quantum graphs}. The transport processes (or \emph{flows}) on the edges are given by  partial differential equations of the form
$\frac{\partial}{\partial t}u_j(t,x)= c_j \frac{\partial}{\partial x}u_j(t,x)$
 and are interlinked in the common nodes via some prescribed transmission conditions. 

Such a problem was considered by Dorn et al.~\cite{Dorn2008,DKNR2010,DKS2009} on the state space $\L^1\left(\left[0,1\right],\ell^1\right)$  applying the theory of strongly continuous operator semigroups. 
A semigroup approach to flows in finite metric graphs was first presented by Kramar and Sikolya \cite{KS2005} and further used in \cite{MS2007,Vertex2008,DKS2009,BDK2013,BanFal2015,Positive2017}  while
transport processes in infinite networks were also studied in  \cite{BanFal2017,BanPuc2018}.
All these results were obtained in the $\L^1$-setting. {M\'atrai and Sikolya \cite{MS2007} went further and  considered the space of bounded Borel measures. When studying} problems in infinite graphs, the flow problem in the $\L^{\infty}$-setting should be interesting for applications as well. Von Below and Lubary \cite{vBeLu2005,vBeLu2009}, for example, study eigenvalues of the Laplacian on infinite networks in an $\L^{\infty}$-setting. {Also the celebrated existence, uniqueness, and stability results for linear transport equations by DiPerna and Lions  \cite{DiPLio1989}  are proven in the $\L^p$-setting for all $1\le p \le\infty$.}
To the best of our knowledge, transport equations on infinite metric graphs with an $\L^{\infty}$-state space have not yet been studied. 
{By Lotz' Theorem \cite[Thm.~3]{Lotz1985}, every strongly continuous semigroup on an $\L^{\infty}$-space is automatically uniformly continuous which means that semigoups on such spaces, which are not strongly continuous, are of special interest and a different approach is needed.}
We consider the transport problem on the state space $\L^{\infty}\left(\left[0,1\right],\ell^1\right)$ where the obtained operator semigroup is  not strongly continuous.
To tackle this we apply the theory of \emph{bi-continuous semigroups} that  was introduced by K\"uhnemund  \cite{KuPHD,Ku} and further developed by Farkas \cite{FaPHD,FaStud,FaSF} and Albanese-Lorenzi-Manco \cite{ALM}.

 This paper is organised as follows. Section 2 is a preliminary section which we start by a short introduction to the theory of bi-continuous semigroups. We also recall certain duality concepts of Bochner $\L^p$-spaces that are needed later to obtain the generation theorem, and introduce some notions for networks and metric graphs.
In Section 3 we present our flow problem for an infinite metric graph. We first prove the well posedness in the case when all flow velocities $c_j$ equal 1. Next, we generalise this result to the case with rationally dependent velocities satisfying a finiteness condition. Finally, we show that the general problem on a finite metric graph is well-posed.

\section{Preliminaries}

\subsection{Bi-Continuous Semigroups}
A family of linear bounded operators  $(T(t))_{t\geq0}$ on a Banach space $X$ is called a (one-parameter) \emph{strongly continuous semigroup} if it satisfies the following properties,
\begin{enumerate}
\item $T(0)=I$, $T(t+s)=T(t)T(s)$ for all $t,s\geq 0$, and
\item  the function $t\mapsto T(t)f$ is continuous for all $f\in X$. 
\end{enumerate}
Strongly continuous semigroups and their applications to evolution equations have been studied intensively in the last decades, we refer to monographs {\cite{EN,NagelPos1986,Positive2017} and references therein.
It is well-known, however, that there are important examples of semigroups which fail to satisfy property (2) above, i.e., that are not strongly continuous with respect to the Banach space norm. {A} standard example is the left translation semigroup $(T_{l}(t))_{t\geq0}$ on $\BC(\RR)$ defined by 
$T_{l}(t)f(x)=f(x+t),$ $ t\geq0. $
It is also known that this semigroup is strongly continuous with respect to the so-called compact open topology $\tau_{co}$. This is a locally convex topology induced by the family of seminorms $\semis=\{p_K\mid  K\subseteq\RR\ \text{compact}\}$ where
\[
p_K(f)=\sup_{x\in K}{\left|f(x)\right|},\quad f\in\BC(\RR).
\] 
Hence, it might be useful to equip the given Banach space $\left(X,\left\|\cdot\right\|\right)$ with an additional locally convex topology $\tau$. This is the general idea of the so-called bi-continuous semigroups. 
Before giving the proper definition we state the main assumptions on the interplay between the norm and the locally convex topology $\tau$. 

\begin{assumption}\label{ass:bicontassump}
\begin{iiv}
	\item $\tau$ is a Hausdorff topology and is coarser then the norm-topology on $X$, i.e., the identity map $(X,\left\|\cdot\right\|)\rightarrow(X,\tau)$ is continuous.
	\item $\tau$ is sequentially complete on norm-bounded sets, i.e., every $\left\|\cdot\right\|$-bounded $\tau$-Cauchy sequence in $\tau$-convergent.
	\item The dual space of $(X,\tau)$ is norming for $X$, i.e.,
	\[
	\left\|f\right\|=\sup_{\substack{{\varphi\in(X,\tau)',  \left\|\varphi\right\|\leq1}}}{\left|\varphi(f)\right|}.
	\]
\end{iiv}
\end{assumption}

\begin{remark}
\begin{abc}
	\item One can re-formulate the third assumption by the following equivalent statement: There is a set $\semis$ of $\tau$-continuous seminorms defining the topology $\tau$, such that
\begin{equation*}\label{eq:semisnorm}
\|f\|=\sup_{p\in\semis}p(f).
\end{equation*}
	\item The above mentioned compact-open topology $\tau_{co}$  on $\BC(\RR)$ satisfies all these assumptions.
\end{abc}
\end{remark}

Now we are in the state to formulate the definition of a bi-continuous semigroup.

\begin{definition}[K\"uhnemund \cite{Ku}]\label{def:bicontdef}
Let $X$ be a Banach space with norm $\|\cdot\|$ together with a locally convex topology $\tau$, such that the conditions in Assumption \ref{ass:bicontassump} are satisfied. We call {a family of linear bounded operators} $(T(t))_{t\geq0}$ a \emph{bi-continuous semigroup} on $X$ if the following holds.
\begin{iiv}
\item $(T(t))_{t\geq0}$ satisfies the \emph{semigroup property}, i.e., $T(t+s)=T(t)T(s)$ and $T(0)=I$ for all $s,t\geq 0$.
\item $(T(t))_{t\geq0}$ is \emph{strongly $\tau$-continuous}, i.e., the map $\varphi_f:[0,\infty)\to(X,\tau)$ defined by $\varphi_f(t)=T(t)f$ is continuous for every $f\in X$.
\item $(T(t))_{t\geq0}$ is \emph{exponentially bounded}, i.e., there exist $M\geq 1$ and $\omega\in\RR$ such that $\left\|T(t)\right\|\leq M\ee^{\omega t}$ for each $t\geq0$.
\item $(T(t))_{t\geq0}$ is \emph{locally-bi-equicontinuous}, i.e., if $(f_n)_{n\in\NN}$ is a norm-bounded sequence in $X$ which is $\tau$-convergent to $0$, then also $(T(s)f_n)_{n\in\NN}$ is $\tau$-convergent to $0$ uniformly for $s\in[0,t_0]$ for each fixed $t_0\geq0$.
\end{iiv}
\end{definition}

The \emph{growth bound} of $(T(t))_{t\geq0}$ is defined as 
\[\omega_0(T):=\inf\{\omega\in \RR \mid \text{ there exists } M\ge 1 \text{ such that }\left\|T(t)\right\|\leq M\ee^{\omega t} \text{ for all } t\geq0\}.\]

As in the case of strongly continuous semigroups one can define a generator and relate it to the well-posedness property of abstract initial value problems. 
The \emph{generator} $A$ of {a} bi-continuous semigroup $(T(t))_{t\geq0}$ is defined as
\[Af:=\tlim_{t\to0}{\frac{T(t)f-f}{t}}\] with the domain 
\[\dom(A):=\left\{f\in X \mathrel{\bigg|}  \tlim_{t\to0}{\frac{T(t)f-f}{t}}\ \text{exists and} \ \sup_{t\in(0,1]}{\frac{\|T(t)f-f\|}{t}}<\infty\right\}.\]

Let us recall some more notions from the bi-continuous setting. A subset $M\subseteq X$ is called \emph{bi-dense} if for every $f\in X$ there exists a $\left\|\cdot\right\|$-bounded sequence $(f_n)_{n\in\NN}$ in $M$ which is $\tau$-convergent to $f$. An operator $A$ is called \emph{bi-closed}, whenever for $f_n\stackrel{\tau}{\to}f$ and $Af_n\stackrel{\tau}{\to}g$, where both sequences are norm-bounded, it holds $f\in\dom(A)$ and $Af=g$.

We have collected here some basic properties of generators of  bi-continuous semigroups. For the proofs we refer to \cite{KuPHD,Ku,FaStud}.

\begin{proposition}\label{thm:bicontprop}
The following assertions hold for the generator $(A,D(A))$ of a bi-continuous semigroup $(T(t))_{t\geq0}$. 
\begin{abc}
\item {The} operator $A$ is {bi-closed} and its domain $\dom(A)$ is {bi-dense} in $X$. 
\item For $f\in\dom(A)$ one has $T(t)f\in\dom(A)$ and $T(t)Af=AT(t)f$ for all $t\geq0$.
\item For $t>0$ and $f\in X$ one has 
\[ \int_0^t{T(s)f\ \dd s}\in\dom(A)\ \ \text{and}\ \ A\int_0^t{T(s)f\ \dd s}=T(t)f-f. \]
\item For $\lambda>\omega_0(T)$ one has $\lambda\in\rho(A)$ and  for every $f\in X$ \
\begin{align}\label{eq:bicontlaplace}
R(\lambda,A)f=\int_0^{\infty}{\ee^{-\lambda s}T(s)f\ \dd s},\end{align}  where the integral is a $\tau$-improper integral.
\item The semigroup $(T(t))_{t\geq0}$ is uniquely determined by its generator $(A,D(A))$.
\end{abc}
\end{proposition}

The following generalisation of the classical notion of well-posedness is due to Farkas \cite[Def.~4.1.1]{FaPHD}. By ${\mathrm{B_{loc}}}(\RR_+,X)$ we denote the space of functions that are bounded on each compact subset of $\RR_+$ and the differentiation is understood in the vector valued sense with respect to $\tau$.
\begin{definition}
The abstract Cauchy problem
 \begin{align}\label{eqn:ACP}\tag{ACP}
\begin{cases}
\dot{u}(t)=Au(t),&\quad t\geq0,\\
u(0)=f \in \dom(A),
\end{cases}
\end{align}
is \emph{well-posed} in $X$ if
\begin{iiv}
\item for every $f\in\dom(A)$ there exists a solution $u(t):=u(t,f)$ of \eqref{eqn:ACP} with $u\in {\mathrm{B_{loc}}}(\RR_+,X)\cap\mathrm{C}^1(\RR_+,(X,\tau))$ and $\dot{u}\in {\mathrm{B_{loc}}}(\RR_+,X)$, 
\item the solution is unique, and
\item the solution $u$ of \eqref{eqn:ACP} depends continuously on the initial data $f$, i.e., if the sequence $(f_n)_{n\in\NN}$ is norm-bounded and $\tau$-convergent to $0$ then the solutions $u_n(t):=u_n(t,f_n)$ converge to $0$ in $\tau$ and uniformly on compact intervals $\left[0,t_0\right]$.
\end{iiv}
\end{definition}

\begin{theorem}{\cite[Thm.~4.1.2]{FaPHD}}\label{thm:wp}
If $(A,\dom(A))$ generates a bi-continuous semigroup then the abstract Cauchy problem \eqref{eqn:ACP} is well-posed.\end{theorem}

This result is very useful for applications in combination with some version of the generation theorem for bi-continuous semigroups. We only recall here a variant of Trotter--Kato approximation theorem, see \cite{KuPHD,AM2004,FaPHD}. For that we also evoke the notion of uniformly bi-continuous semigroups \cite[Def.~ 2.1]{KuPHD}.

\begin{definition}\label{def:UniBiCont}
Let $(T_n(t))_{t\geq0}$, $n\in\NN$, be bi-continuous semigroups on $X$. They are called \emph{uniformly bi-continuous} (of type $\omega$) if the following conditions hold.
\begin{iiv}
	\item There exist $M\geq1$ and $\omega\in\RR$ such that $\left\|T_n(t)\right\|\leq M\ee^{\omega t}$ for all $t\geq0$ and $n\in\NN$.
	\item $(T_n(t))_{t\geq0}$ are locally bi-equicontinuous uniformly for $n\in\NN$, i.e., for every $t_0\geq0$ and for every $\left\|\cdot\right\|$-bounded sequence $(f_k)_{k\in\NN}$ in $X$ which is $\tau$-convergent to $0$ we have
	\[
	\tlim_{k\to\infty}{T_n(t)f_k}=0,
	\]
	uniformly for $0\leq t\leq t_0$ and $n\in\NN$.
\end{iiv}
\end{definition}

Let us also recall the notion of a bi-core, as defined in \cite[Def.~1.20]{KuPHD}.
A subspace $D$ of the domain of a linear operator $(A,\dom(A))$ on a Banach space $X$ is a \emph{bi-core} for $A$ if for all $f\in\dom(A)$ there exists a sequence $(f_n)_{n\in\NN}$ in $D$ such that $(f_n)_{n\in\NN}$ and $(Af_n)_{n\in\NN}$ are $\left\|\cdot\right\|$-bounded and $\lim_{n\to\infty}{f_n}=f$ with respect to the locally convex (graph) topology $\tau_A$ generated by the family of seminorms
\[
\semis_A:=\left\{p(\cdot)+q(A\cdot) \mid p,q\in\semis\right\}.
\]

Having these definitions in mind, we can {state} the Trotter--Kato approximation theorem for bi-continuous semigroups which was first proven by K\"uhnemund \cite[Thm.~2.6]{KuPHD} and later on in a more general version by Albanese and Mangino \cite[Thm.~3.6]{AM2004}. For locally equicontinuous semigroups on locally convex space this was done by Albanese and K\"uhnemund \cite[Thm.~16]{AK2002}.

\begin{theorem}{\cite[Thm.~2.6]{KuPHD}}\label{thm:TK}
Let $(T_n(t))_{t\geq0}$,  $n\in\NN$,  be  bi-continuous semigroups with generators $(A_n,\dom(A_n))$ such that they are uniformly bi-continuous of type $\omega$ and let $\lambda_0>\omega$. Consider the following assertions.
\begin{abc}
	\item There exists a bi-densely defined operator $(A,\dom(A))$ such that $A_nf\stackrel{\left\|\cdot\right\|}{\rightarrow} Af$ for all $f$ in a bi-core of $A$ and such that $\mathrm{Ran}(\lambda_0-A)$ is bi-dense in $X$. 
	\item There exists an operator $R\in\LLL(X)$ such that $R(\lambda_0,A_n)f\stackrel{\left\|\cdot\right\|}{\rightarrow}Rf$ for all $f$ in a subset of $\mathrm{Ran}(R)$ which is bi-dense in $X$.
	\item There exists a bi-continuous semigroup $(T(t))_{t\geq0}$ with generator $(B,\dom(B))$ such that $T_n(t)f\stackrel{\tau}{\to}T(t)f$ for all $f\in X$ uniformly for $t$ in compact intervals
\end{abc}
Then the implications $\mathrm{(a)}\Rightarrow\mathrm{(b)}\Rightarrow\mathrm{(c)}$ hold. In this case, {the bi-closure $\overline{A}^{\tau}$ of $A$ equals $B$}.
\end{theorem}

\begin{remark}\label{rem:TK}
In the proof of \cite[Thm.~2.6]{KuPHD} one observes
that operator $R$ in assertion (b) gives rise to a pseudo-resolvent that is used to define operator $(B,\dom(B))$ in assertion (c).
\end{remark}

We refer to \cite{KuPHD,Ku,FaPHD,BF,FaStud,AM2004} for further properties of bi-continuous semigroups and their generators.

\subsection{{Bochner $\L^p$-spaces, duality, and translation semigroups}}\label{sec:Lp-dual}

We consider spaces of the form $\L^p\left([0,1],X\right)$ where $X$ is a Banach space. From \cite[Chapter IV,Section 1]{DiestelUhl1977} we obtain that for $1<p<\infty$ and $\frac{1}{p} + \frac{1}{q} =1$ one has
\begin{align}\label{eqn:DualL1}
\L^p\left(\left[0,1\right],X\right)'\cong\L^q\left(\left[0,1\right],X'\right) \quad\text{ and }\quad
\L^1\left(\left[0,1\right],X\right)'\cong\L^{\infty}\left(\left[0,1\right],X'\right)
\end{align}
whenever $X'$ has the {Radon--Nikodym property}. 
If this is not the case, one only has an isometric inclusion $\L^q([0,1],X')\hookrightarrow\L^p([[0,1],X)'$ for $1\leq p<\infty$.
{Recall, that the pairing is defined by
\begin{equation}\label{eqn:weak*top} 
\left\langle f,g\right\rangle:=\int_0^1{\left\langle f(s),g(s)\right\rangle_X \dd{s}},\quad f\in 
\L^q\left(\left[0,1\right],X'\right),\ g\in\L^p\left(\left[0,1\right],X\right),
\end{equation}
and $\left\langle \cdot,\cdot\right\rangle_X$ denotes the dual pairing between $X$ and $X'$.}

\begin{example}\label{ex:Lp-dual} It is known that the space $\ell^1$ has the Radon-Nikodym property while the spaces $\c_0$, $\c$, and $\ell^{\infty}$ do not.
Recall that $(\ell^1)' = \ell^{\infty}$ while 
$(\c_0)'\cong\ell^1$ as well as $(\c)'\cong\ell^1$ (that is, {the} space $\ell^1$ does not have a unique predual space).
By \eqref{eqn:DualL1} we obtain that 
\[
\L^{\infty}\left(\left[0,1\right],\ell^1\right)\cong \L^1\left(\left[0,1\right],\c_0\right)' \cong \L^1\left(\left[0,1\right],\c\right)'.
\]
{{while  $\L^{\infty}\left(\left[0,1\right],\ell^{\infty}\right)$  is only isomorphic to a subspace of $\L^1\left(\left[0,1\right],\ell^1\right)'$.}}
\end{example}

{We would like to study} the left translation semigroup $(T_l(t))_{t\geq0}$ on the space {$X=\L^{\infty}\left(\left[0,1\right],\ell^1\right)$}, 
{which is} not strongly continuous. By using duality arguments we will show that it is, however, a bi-continuous semigroup on $X$. 
{We start by showing } that the right translation semigroup on $\L^{1}\left(\left[0,1\right],\c_0\right)$ is {strongly continuous.}

\begin{lemma}\label{prop:TransL1}
The right translation semigroup $(T_r(t))_{t\geq0}$, defined by 
\[
T_r(t)f(s):=\begin{cases}f(s-t),&\quad s-t\geq0,\\0,&\quad s-t<0,\end{cases}
\]
for $t\geq0$ and $s\in\left[0,1\right]$, is strongly continuous on $\L^{1}\left(\left[0,1\right],\c_0\right)$.
\end{lemma}

\begin{proof}
Let $\textbf{x}=(x_n)_{n\in\mathbb{N}}\in\c_0$ and let $f:=\textbf{x}\cdot\chi_{\Omega}$ for a measurable subset $\Omega\subseteq\left[0,1\right]$. We first show that $T_{\mathrm{r}}(t)f\rightarrow f$ with respect to the norm on $\L^1\left(\left[0,1\right],\c_0\right)$ as $t\to 0$:
\begin{align*}                                                                             
\int_0^1{\left\|T_r(t)f(s)-f(s)\right\|_{\c_0} \dd{s}}&=\int_0^1{\left\|f(s-t)\chi_{\left[s-t\geq0\right]}-f(s)\right\|_{\c_0} \dd{s}}\\
&=\left\|\textbf{x}\right\|_{\c_0}\cdot\int_0^1{\left|\chi_{\Omega}(s-t)-\chi_{\Omega}(s)\right| \dd{s}}\\
&=\left\|\textbf{x}\right\|_{\c_0}\cdot\int_0^1{\left|\chi_{(\Omega+t)\Delta\Omega}(s)\right| \dd{s}}\\
&=\left\|\textbf{x}\right\|_{\c_0}\cdot\lambda^1\left((\Omega+t)\Delta\Omega\right) \to 0  \text{ as } {t\to 0},
\end{align*}
where $\lambda^1$ is the one-dimensional Lebesgue measure on $\left[0,1\right]$ and $\Delta$ the symmetric difference of sets defined by $A\Delta B:=(A\cup B)\setminus(A\cap B)$. Since every function $f\in\L^{1}\left(\left[0,1\right],\c_0\right)$ is an increasing limit of linear combinations of functions of the form $\textbf{x}\cdot\chi_{\Omega}$ for some $\textbf{x}\in\c_0$ and measurable set $\Omega\subseteq\left[0,1\right]$,  the vector-valued version of Beppo--Levis's monotone convergence theorem {\cite[Prop.~2.6]{vanZuilen}} yields the result.
\end{proof}

{Now note,} that the left translation semigroup on {$\L^{\infty}\left(\left[0,1\right],\ell^1\right) =\L^1\left(\left[0,1\right], \c_0 \right)' $} is the adjoint semigroup of the right translation semigroup on $\L^1\left(\left[0,1\right], \c_0 \right)$, see Subsection \ref{sec:Lp-dual} and \cite[II.5.14]{EN}.

\begin{lemma}\label{prop:TransL2}
The left translation semigroup $(T_l(t))_{t\geq0}$, defined by 
\[
T_l(t)f(s):=\begin{cases}f(s+t),&\quad s+t\leq 1,\\0,&\quad s+t>1,\end{cases}
\]
for $t\geq0$ and $s\in\left[0,1\right]$,
 is bi-continuous on  $\L^{\infty}\left(\left[0,1\right],\ell^1\right)$ with respect to the weak$^*$-topology.
\end{lemma}

\begin{proof}
By \cite[Corollary 2.1.7]{Scirrat2016}, the dual space $Y'$  of any Banach space $Y$ satisfies Assumption \ref{ass:bicontassump} for the weak$^*$-topology.
Moreover, by \cite[Proposition 3.18]{KuPHD}, the dual semigroup $(T'(t))_{t\geq0}$ on $Y'$ is bi-continuous with respect to the weak$^*$-topology whenever $(T(t))_{t\geq0}$ is a strongly continuous semigroup on $Y$.  Thus, \cite[Example 1.3]{Farkas2011}, Example \ref{ex:Lp-dual} and  Lemma \ref{prop:TransL1} imply the result.
 \end{proof}

\begin{remark}
{Let us stress here that, given a strongly continuous semigroup, its adjoint semigroup on the dual space is not necessarily strongly continuous. Indeed, the translation semigroup is not strongly continuous  in the sup-norm of the $\L^{\infty}$-space and the largest subspace where it is strongly continuous  is $\mathrm {C_{ub}}$,  the space of all bounded, uniformly continuous functions (the so-called \emph{sun dual}, see \cite[Sec.~II.2.5]{EN}). Since we are interested  in the solutions taking  place in $X=\L^{\infty}\left(\left[0,1\right],\ell^1\right)$, the classical theory on adjoint semigroups does not apply. The above result, however, shows that the theory of bi-continuous semigroups is appropriate here.}
\end{remark}

\subsection{Infinite Networks, Metric Graphs}\label{subsec:Networks}

We use the notation introduced in \cite{KS2005} for finite {networks} and expanded in \cite{Dorn2008} to infinite networks. {A} network is {modeled} with an 
\emph{infinite directed graph}  $G= (V,E)$ with a  set of \emph{vertices} $V=\left\{\vv_i\mid i\in I\right\}$ and a set of \emph{directed edges} $E=\left\{\ee_j\mid  j\in J\right\}\subseteq V\times V$ for some countable sets $I,J\subseteq\NN$.
For a directed edge $\ee=(\vv_i,\vv_k)$ we call $\vv_i$ the \emph{tail} and  $\vv_k$ the \emph{head} of $\ee$. Further, the edge $\ee$ is an \emph{outgoing edge} of the vertex $\vv_i$ and an \emph{incoming edge} for the vertex $\vv_k$. 
We assume that graph $G$ is \emph{simple}, i.e., there are no loops or multiple edges, and \emph{locally finite}, i.e., each vertex only has finitely many {incident} edges. 

Graph  $G$ is \emph{weighted}, that is equipped with some weights $0 \le  w_{ij} \le 1$ such that
\begin{equation}\label{eqn:noAbsorb}
\sum_{i\in J}{w_{ij}}=1 \text{ for all } j\in J.
\end{equation}
The structure of a graph can be described by its incidence and/or adjacency matrices. We shall only use the so-called \emph{weighted (transposed) adjacency matrix of the line graph} $\mathbb{B} = (\mathbb{B}_{ij})_{i,j\in J}$ defined as
\begin{equation}\label{eqn:adjMat}
\mathbb{B}_{ij}:=\begin{cases}
w_{ij}& \text{ if }\stackrel{\ee_j}{\longrightarrow}\vv\stackrel{\ee_i}{\longrightarrow},\\
0 & \text{ otherwise}.
\end{cases}
\end{equation}
By \eqref{eqn:noAbsorb},  matrix $\mathbb{B}$ {defines a stochastic operator on $\ell^1$. From that it follows that} $r(\mathbb{B})= \|\mathbb{B}\| =1$. It reflects many properties of graph $G$. For example, $\mathbb{B}$ is irreducible iff graph $G$ is strongly connected (see \cite[Prop.~4.9]{Dorn2008}).

We identify every edge of our graph with the unit interval,  $\ee_j\equiv\left[0,1\right]$ for each $j\in J$, and parametrise it contrary to its direction, so that it is assumed to have its tail at the endpoint $1$ and its head at the endpoint $0$. For simplicity we use the notation $\ee_j(1)$ and $\ee_j(0)$ for the tail and the head, respectively. In this way we obtain a \emph{metric graph}.

For the unexplained terminology we refer to \cite[Sect.~18]{Positive2017} and \cite{Dorn2008}.

\section{Transport problems in (in)finite metric graphs}

We now consider a transport process (or a flow) along the edges of an infinite network, {modeled} by a metric graph $G$.  The distribution of material along {the} edge $\ee_j$ at time $t \ge0$ is described by {a} function $u_j(x,t)$ for $x\in [0, 1]$. The material is transported along {the} edge $\ee_j$ with constant velocity $c_j>0$ {and is absorbed according to the absorption rate $q_j(x)$.} We assume that {$q\in\L^{\infty}\left(\left[0,1\right],\ell^{\infty}\right)$ and }
\begin{equation}\label{as:c-lim}
0< c_{\min} \le c_j \le c_{\max} < \infty
\end{equation}
for all $j\in J$.  Let  $C:=\diag(c_j)_{j\in J}$ be a diagonal velocity matrix and define another weighted adjacency matrix of the line graph by 
\begin{equation*}\mathbb{B}^C:=C^{-1}\mathbb{B} C.\end{equation*}

In the vertices the material
gets redistributed according to some prescribed rules. This is modelled in the boundary conditions by using the adjacency matrix $\mathbb{B}_C$.
The flow process on $G$ is thus given by the following infinite system of equations
\begin{align}\label{eqn:F}
\begin{cases}
\frac{\partial}{\partial t}u_j(x,t)=c_j\frac{\partial}{\partial x}u_j(x,t) {+ q_j(x)u_j(x,t)},&\quad x\in\left(0,1\right),\ t\geq0,\\
u_j(1,t)= \sum_{k\in J} \mathbb{B}^C_{jk} u_k(0,t),&\quad t\geq0,\\
u_j(x,0)= f_j(x),&\quad x\in\left(0,1\right),
\end{cases}
\end{align}
for every $j\in J$, where $f_j(x)$ are the initial distributions along the edges.

One can give different interpretations to the weights $w_{ij}$, i.e, entries of the matrix $\mathbb{B}$, resulting in different transport problems. The two most obvious  are the following.
\begin{enumerate}
\item $w_{ij}$ is the proportion of the material arriving from edge $\ee_j$ leaving on edge $\ee_i$.
\item  $w_{ij}$ is the proportion of the material arriving from vertex $\ee_j(0) = \ee_i(1)$ leaving on edge $\ee_i$.
\end{enumerate}
Note, that in both situations \eqref{eqn:noAbsorb}  represents a conservation of mass and the assumption on local finiteness of the graph guarantees that all the sums are finite.
While the latter situation is the most common one (see e.g.~\cite{Dorn2008,KS2005,Positive2017}) 
the first one was considered for finite networks in \cite[Sect.~5]{BDK2013}. Here, we will not give any particular interpretation and will treat all the cases simultaneously.
\begin{remark}
By replacing in \eqref{eqn:F} the graph matrix $\mathbb{B}^C$ with some other matrix, one obtains a more general initial-value problem that does not necessarily consider a process in a physical network. Such a problem from population dynamics was  for example studied in \cite{BanFal2017}.
Furthermore, a question when can such a general problem be identified with a corresponding problem on a metric graph was raised in \cite{BanFal2015}.
\end{remark}

{
\begin{remark} 
Note, that we could also consider space-dependent velocities $c_j(x)$ and then either re-norm the space as in \cite{MS2007}:
\[
\|f\|_C:=  \sum_{j\in J} \int_0^1 \frac{|f_j(s)|}{c_j(s)} \dd{s}, \quad f\in \L^{\infty}\left(\left[0,1\right],\ell^1\right) \]
 or change the variables appropriately to obtain an equivalent problem with constant velocities.
\end{remark}
}

\subsection{The simple case}
\label{subs:simple}

First we assume that all the velocities are the same {and there is no absorption}: $c_j=1$ {and $q_j(\cdot) = 0$} for each $j\in J$. 
As the state space we {choose} $X:=\L^{\infty}\left(\left[0,1\right],\ell^1\right)$ equipped with the norm
\[
\left\|f\right\|_X:=\esssup_{s\in\left[0,1\right]}{\left\|f(s)\right\|_{\ell^1}}.
\] 
On {the} Banach space $X$ we define {the} operator $(A,\dom(A))$ by 
\begin{align}\label{eqn:Gen1}
\begin{split}
A&:=\diag \left(\frac{\dd}{\dd{x}} \right)\\
\dom(A)&:=\left\{v\in\W^{1,\infty}\left(\left[0,1\right],\ell^1\right)\mid  v(1)=\mathbb{B}v(0)\right\}
\end{split}
\end{align}
Observe that the corresponding abstract Cauchy problem 
\begin{align}\label{eqn:ACPn}
\begin{cases}
\dot{u}(t)=Au(t),&\quad t\geq0,\\
u(0)=(f_j)_{j\in J},&
\end{cases}
\end{align}
on $X$ is equivalent to the flow problem \eqref{eqn:F} in case when all the velocities equal 1. This problem was considered by Dorn \cite{Dorn2008} on the state space $\L^1\left(\left[0,1\right],\ell^1\right)$ where an explicit formula for the solution semigroup  in terms of a shift and matrix $\mathbb{B}$ was derived. {We will use this formula and} 
show that it yields a bi-continuous semigroup on $L^{\infty}\left(\left[0,1\right],\ell^1\right)$. For that we have to check all the assertions from Definition \ref{def:bicontdef} which we do in several steps.

\begin{lemma}\label{prop:StrCont}
The semigroup $(T(t))_{t\geq0}$ on $X=L^{\infty}\left(\left[0,1\right],\ell^1\right)$, defined by 
\begin{align}\label{eqn:seminet}
T(t)f(s)=\mathbb{B}^nf(t+s-n),\quad n\leq t+s<n+1,\ f\in X,\ n\in\mathbb{N}_0,
\end{align}
is strongly continuous with respect to the weak$^*$-topology.
\end{lemma}

\begin{proof}
The semigroup property is easy to verify.
Observe that for any $f\in X$, $g\in\L^1\left(\left[0,1\right],\c_0\right)$, and $t\in\left(0,1\right]$ we have
\begin{align*}
&\left|\left\langle T(t)f-f,g\right\rangle\right| =\\
&=\left|\int_0^1{\left\langle T(t)f(s)-f(s),g(s)\right\rangle \dd{s}}\right|\\
&\leq\left|\int_0^{1-t}{\left\langle f(s+t)-f(s),g(s)\right\rangle \dd{s}}\right|+\int_{1-t}^1{\left|\left\langle \mathbb{B}f(s+t-1)-f(s),g(s)\right\rangle\right| \dd{s}}\\
&\leq\left|\int_0^1{\left\langle T_l(t)f(s)-f(s),g(s)\right\rangle \dd{s}}\right|+\int_{1-t}^1{\left|\left\langle \mathbb{B}f(s+t-1)-f(s),g(s)\right\rangle\right| \dd{s}}.
\end{align*}
Now, notice that the second summand  vanishes since $\lambda^1\left(\left[1-t,1\right]\right) \to 0 $ as $t\to 0$. Here, $\lambda^1$ is the one-dimensional Lebesgue measure on the unit interval $\left[0,1\right]$. 
By Lemma \ref{prop:TransL2}, the left translation semigroup is  bi-continuous on $X$, 
which means, in particular, that it is strongly continuous with respect to the weak$^*$-topology and hence the first summand also vanishes as $t\rightarrow0$.
\end{proof}

\begin{lemma}\label{prop:ExpBd}
The semigroup $(T(t))_{t\geq0}$, defined by \eqref{eqn:seminet}, is a contraction semigroup on  $X$.
\end{lemma}

\begin{proof}
Let $f\in X$ and $t\geq0$. Then there exists $n\in\mathbb{N}$ such that $n\leq t<n+1$. This means that for $s\in\left[0,1\right]$ one has $n\leq s+t<n+2$. By \eqref{eqn:seminet}, we can make the following estimate:
\begin{align*}
&\left\|T(t)f\right\|_X=\\
&=\esssup_{s\in\left[0,1\right]}{\left\|T(t)f(s)\right\|_{\ell^1}}\\
&\leq\max\left\{\esssup_{{s\in\left[0,n+1-t\right)}}{\left\|\mathbb{B}^nf(s+t-n)\right\|_{\ell^1}},\ \esssup_{{s\in\left[n+1-t,1\right]}}{\left\|\mathbb{B}^{n+1}f(s+t-n-1)\right\|_{\ell^1}}\right\}.
\end{align*}
Since $\left\|\mathbb{B}^n\right\|=\left\|\mathbb{B}\right\|^n =1 $, we have
\[
\left\|\mathbb{B}^nf(s+t-n)\right\|_{\ell^1}\leq\left\|\mathbb{B}^n\right\|\cdot\left\|f\right\|_{X} =  \left\|f\right\|_{X}
\]
and hence, $\left\|T(t)f\right\|_X \le \left\|f\right\|_X$.
\end{proof}

\begin{lemma}\label{prop:LocBiEquiCont}
The semigroup $(T(t))_{t\geq0}$, defined by \eqref{eqn:seminet}, is locally bi-equicontinuous with respect to the weak$^*$-topology on $X=\L^{\infty}\left(\left[0,1\right],\ell^1\right)$.
\end{lemma}

\begin{proof}
Let $(f_n)_{n\in\mathbb{N}}$ be a sequence of functions in $X$ that is $\left\|\cdot\right\|_{X}$-bounded and converges to $0$ with respect to the weak$^*$-topology. By Definition \ref{def:bicontdef}, we need to show that  $(T(t)f_n)_{n\in\mathbb{N}}$  converges to $0$ with respect to the weak$^*$-topology uniformly on compact intervals $\left[0,t_0\right]$. 
To this end, fix $t_0>0$ For every  $t\in\left[0,t_0\right]$ one can find $k\in\NN_0$, $0\le k \le {\left\lceil t_0\right\rceil}$, such that  
$k\leq s+t {<} k+1$ for all $s\in [0,1]$ and by \eqref{eqn:seminet} we have
\[ 
T(t)f_n(s)= \mathbb{B}^kf_n(s+t-k),
\]
hence
\begin{align*}
\left\langle T(t)f_n,g\right\rangle & = \int_0^1{\left\langle T(t)f_n(s),g(s)\right\rangle \dd{s}}\\
&=\int_0^1{{\left\langle \mathbb{B}^{k}f_n(s+t-k),g(s)\right\rangle \dd{s}}}\\
&=\int_0^1{{\left\langle T_l(t-k){f_n(s)},\left(\mathbb{B}^{k}\right)'g(s)\right\rangle \dd{s}}}
\end{align*}
for any $g\in\L^{1}\left(\left[0,1\right],\c_0\right)$. 
{Observe, that our assumption on locally finiteness of the graph implies that both $\BB$ and $\BB'$ only have finitely many elements in each row or column. Therefore,  $\BB'$ leaves the space $c_0$ invariant.}  
Moreover, there are only finitely many choices of the integer $k\in\left[0,\left\lceil t_0\right\rceil\right]$ and we have
\[ \{T_l(t-k) \mid 0\le t \le t_0\} = \{T_l(\tau) \mid 0\le \tau \le 1\}.\]  
Since, by Lemma \ref{prop:TransL2}, the left translation semigroup $(T_l(t))_{t\geq0}$ is bi-continuous, hence locally bi-equicontinuous {with respect to the weak$^*$-topology on $X$, also $\left\langle T(t)f_n,g\right\rangle$} tends to  $0$ uniformly on $\left[0,t_0\right]$. This finishes the proof.
\end{proof}

Let us here recall the explicit expression of the resolvent of operator $(A,\dom(A))$ defined by \eqref{eqn:Gen1} which was obtained in \cite[Theorem 18]{Dorn2008}. This result does not rely on the Banach space and remains the same by taking $X=\L^{\infty}\left(\left[0,1\right],\ell^1\right)$.
\begin{proposition}\label{prop:res}
For $\mathrm{Re}(\lambda)>0$ the resolvent $R(\lambda,A)$ of the operator $(A,\dom(A))$ defined by \eqref{eqn:Gen1} is given by
\[
(R(\lambda,A)f)(s):=\sum_{k=0}^{\infty}{\ee^{-\lambda k}\int_0^1{\ee^{-\lambda(t+1-s)}\mathbb{B}^{k+1}f(t)\ \dd{t}}}+\int_s^1{\ee^{\lambda(s-t)}f(t)\ \dd{t}},
\]
$f\in X$, $s\in\left[0,1\right]$.
\end{proposition}

We are now in the state {to} prove the first generation theorem.

\begin{theorem}\label{thm:BiContNet}
The operator $(A,\dom(A))$, defined in  \eqref{eqn:Gen1}, generates a contraction  bi-continuous semigroup $(T(t))_{t\geq0}$ on $X$ with respect to the weak$^*$-topology. This semigroup is given by \eqref{eqn:seminet}.\end{theorem}

\begin{proof}
By Lemmas \ref{prop:StrCont}, \ref{prop:ExpBd}, and \ref{prop:LocBiEquiCont}, {the} semigroup  $(T(t))_{t\geq0}$ defined by \eqref{eqn:seminet} is a bi-continuous semigroup with respect to the weak$^*$-topology. It remains to show that $(A,\dom(A))$, given in  \eqref{eqn:Gen1},  is the generator of this semigroup. Let $(C,\dom(C))$ be the  generator of  $(T(t))_{t\geq0}$. For  $f\in\dom(A)$ and $s\in\left[0,1\right]$ we have $T(t)f\in\dom(A)$.  By \eqref{eq:bicontlaplace}, the resolvent of $C$ is the Laplace transform of the semigroup  $(T(t))_{t\geq0}$, that is,
for $\lambda > \omega_0(T)$ we have
\begin{align*}
R(\lambda,C)f(s)&=\int_0^{\infty}{\ee^{-\lambda t}T(t)f(s)\ \dd{t}}\\
&=\int_0^{1-s}{\ee^{-\lambda t}f(t+s)\ \dd{t}}+\sum_{n=1}^{\infty}{\int_{n-s}^{n-s+1}{\ee^{-\lambda t}\mathbb{B}^nf(t+s-n)\ \dd{t}}}\\
&=\int_s^{1}{\ee^{-\lambda(\xi-s)}f(\xi)\ \dd{\xi}}+\sum_{n=1}^{\infty}{\int_{0}^{1}{\ee^{-\lambda(\xi-s-n)}\mathbb{B}^nf(\xi)\ \dd{\xi}}}
\end{align*}
By Proposition \ref{prop:res}, the resolvent operators $R(\lambda,A)$ and $R(\lambda,C)$,  coincide on the bi-dense set $D(A)$, so we may conclude that $C=A$.
\end{proof}

\begin{corollary}
If all $c_j=1$ {and $q_j(\cdot) = 0$}, $j\in J$, the flow problem \eqref{eqn:F}  is well-posed on $X=\L^{\infty}\left(\left[0,1\right],\ell^1\right)$.
\end{corollary}

\begin{remark}
All the obtained results  also hold for finite networks. If $G=(V,E)$ is a finite network with $\left|E\right|=m<\infty$, we have  $\ell^1\left(\left\{1,\ldots,m\right\}\right)\cong\CC^m$, hence we consider our semigroups on the space $X=\L^{\infty}\left(\left[0,1\right],\CC^m\right)$.
\end{remark}

\subsection{The rationally dependent case}\label{subs:rational}

We now consider the case when the velocities $c_j$  appearing in \eqref{eqn:F} are not all equal to $1$ and define on $X:=\L^{\infty}\left(\left[0,1\right],\ell^1\right)$ the operator
\begin{align}\label{eqn:Gen}
\begin{split}
A_C&:= \diag\left(c_j\cdot\frac{\dd}{\dd x}\right),\\
\dom(A_C)&:=\left\{f\in\W^{1,\infty}\left(\left[0,1\right],\ell^1\right)\mid  f(1)=\BB^Cf(0) \right\}.
\end{split}
\end{align}

We assume, however,  that the velocities are linearly dependant over $\QQ$: $\frac{c_i}{c_j}\in\QQ$ for all $i,j\in J$,  with a finite common multiplier, that is, 
\begin{equation}\label{pj}
 \text{ there exists } 0<c\in\RR \text{ such that }\ell_j:=\frac{c}{c_j}\in\NN \text{ for all } j\in J.
\end{equation}
This enables us to use the procedure that was introduced in the proof of \cite[Thm.~4.5]{KS2005} and carried out in detail in \cite[Sect.~3]{BaNam2014}.
We construct a new directed graph $\widetilde{G}$ by adding $\ell_j-1$ vertices on edge $\ee_j$ for all $j\in J$. The newly obtained edges inherit the direction of the original edge and are parametrised as unit intervals  $[0,1]$. We can thus consider a new problem on $\widetilde{G}$ with corresponding functions $\widetilde{u}_j$ and velocities $\widetilde{c}_j:=c$ for each $j\in\widetilde{J}$. After  appropriately correcting the initial and boundary conditions the new problem is equivalent to the original one. Since all the velocities on the edges of the new graph are equal, we can treat this case by rescaling to 1 and use the results from Subsection \ref{subs:simple}. Moreover, since  \eqref{as:c-lim} and  \eqref{pj} hold, the procedure described in  \cite[Sect.~3]{BaNam2014} for the finite case  can be as preformed in the infinite case as well. Hence,  we even obtain  an isomorphism between the corresponding semigroups.

\begin{theorem}\label{thm:BiContNet-Q} Let the assumptions \eqref{as:c-lim} and \eqref{pj} on the velocities $c_j$ hold. Then {the}
operator $(A_C,\dom(A_C))$, defined in  \eqref{eqn:Gen}, generates a contraction  bi-continuous semigroup $(T_C(t))_{t\geq0}$ on $X$ with respect to the weak$^*$-topology.
Moreover, there exists {a lattice} isomorphism $S\colon X\to X$ such that
\begin{equation}\label{isomorph}
T_C(ct) f = ST(t)S^{-1} f,
\end{equation}
where the semigroup  $(T(t))_{t\geq0}$  is given by \eqref{eqn:seminet}.
\end{theorem}

Finally, we also allow nonzero absorption functions $q_j$. Notice, that \eqref{eqn:F} is equivalent to the abstract Cauchy problem associated to the following operator
\begin{align}\label{eqn:OpPertProb}
\begin{split}
A&:= \diag\left(c_j\cdot\frac{\dd}{\dd x}+M_{q_j}\right),\\
\dom(A)&:=\left\{f\in\W^{1,\infty}\left(\left[0,1\right],\ell^1\right)\mid  f(1)=\BB^Cf(0) \right\},
\end{split}
\end{align}
where $M_{q_j}$ denotes the multiplication operator with the function $q_j$. We can split the operator $A$ into a sum of operators 
\[
A=\diag\left(c_j\cdot\frac{\dd}{\dd x}\right)+\diag\left(M_{q_j}\right)=:A_C+M_q.
\]
We already know that {the} operator $(A_C,\dom(A))$ generates a bi-continuous semigroup on $X$. In order to show that $(A,\dom(A))$ is a generator as well we shall apply the perturbation theory for bi-continuous semigroups.
\begin{theorem}\label{thm:BiContNet-pert}
Let the assumptions \eqref{as:c-lim} and \eqref{pj} hold and let $q\in\L^{\infty}\left(\left[0,1\right],\ell^{\infty}\right)$.
Then the operator $(A,\dom(A))$, defined by \eqref{eqn:OpPertProb}, generates a bi-continuous semigroup on  $X$ with respect to the weak$^*$-topology.
\end{theorem}

\begin{proof}
By Theorem \ref{thm:BiContNet-Q}, $(A_C,\dom(A))$ is the generator of a bi-continuous semigroup.  Since  $q\in\L^{\infty}\left(\left[0,1\right],\ell^{\infty}\right)$, {the} multiplication operator $M_q$ is a bounded operator on $\L^1\left(\left[0,1\right],c_0\right)$ and its dual $M_q'= M_q$ is again a multiplication operator. Therefore, we can apply the special case of the bounded perturbation theorem for bi-continuous semigroups due to Farkas, cf. \cite[Prop.~4.3]{FaStud} yielding that $(A,\dom(A))=(A_C+M_q,\dom(A))$ generates a bi-continuous semigroup on $X=\L^1\left(\left[0,1\right],c_0\right)'$.
\end{proof}

\begin{remark} Note, that the bi-continuous semigroup obtained in Theorem \ref{thm:BiContNet-pert} is given by the Dyson-Phillips series, see \cite[Thm.~3.5]{FaStud}. The semigroup can also be represented by means of the Trotter product formula for bi-continuous semigroups, cf. \cite[Cor.~4.2]{AM2004} or \cite[Cor.~2.11]{KuPHD}. Let { $(S(t))_{t\geq0}$ denote the semigroup generated by $(A,\dom(A))$. Then} \[
S(t)x=\tlim_{n\to\infty}{\left[{T_C}\left(\tfrac{t}{n}\right)\ee^{\frac{tM_q}{n}}\right]^nx},
\]
for all $x\in X$ and uniformly on compact intervals, {where we denote by $(T_C(t))_{t\geq0}$} the bi-continuous semigroup generated by {$(A_C,\dom(A))$ given in \eqref{isomorph}}..
\end{remark}

\begin{corollary}
If  the assumptions \eqref{as:c-lim} and \eqref{pj}  on the velocities $c_j$  hold  {and $q\in\L^{\infty}\left(\left[0,1\right],\ell^{\infty}\right)$,} {then} the flow problem \eqref{eqn:F}  is well-posed on $X=\L^{\infty}\left(\left[0,1\right],\ell^1\right)$.
\end{corollary}

\subsection{The general case for finite networks}\label{sec:GenCaseFinNet}

\medskip
We finally consider the case of general $c_j\in\RR$ but restrict ourselves to 
 \emph{finite} graphs, i.e., we work on the Banach space $X=\L^{\infty}\left(\left[0,1\right],\CC^m\right)$, where $m$ denotes the number of edges in the graph.
In  \cite[Cor.~18.15]{Positive2017} the Lumer--Phillips generation theorem for positive strongly continuous semigroups is applied to show
that the transport problem with general $c_j\in\RR$ is well-posed on $X=\L^{1}\left(\left[0,1\right],\CC^m\right)$. Since an appropriate variant of this result for the bi-continuous situation is not known,  we proceed differently and use the variant  of Trotter--Kato approximation theorem given in Theorem \ref{thm:TK}.

Let
\[
E_{\lambda}(s):=\diag\left(\ee^{\left({\lambda}/{c_j}\right) s}\right),\quad s\in\left[0,1\right],
\quad
\text{and}\quad \mathbb{B}^{C}_{\lambda}:=E_{\lambda}(-1)\mathbb{B}^{C}.
\]
By using this notation one can write an explicit expression for the resolvent of {the} operator $A_C$ defined in \eqref{eqn:Gen}.  
\begin{lemma}{\cite[Prop.~18.12]{Positive2017}}\label{prop:ExRes}
For $\mathrm{Re}(\lambda)>0$ the resolvent $R(\lambda,A_C)$ of {the} operator $A_C$,  given in \eqref{eqn:Gen}, equals
\begin{align*}
R(\lambda,A_C)=\left(\Id_X+E_{\lambda}(\cdot)\left(1-\BB^{C}_{\lambda}\right)^{-1}\BB^{C}_{\lambda}\otimes\delta_0\right)R_{\lambda},
\end{align*}
where $\delta_0$ denotes the point evaluation at $0$ and
\begin{align*}
\left(R_{\lambda}f\right)(s)=\int_{s}^1{E_{\lambda}(s-t)C^{-1}f(t)\ \dd{t}},\quad s\in\left[0,1\right],\ f\in\L^{\infty}\left(\left[0,1\right],\CC^m\right).
\end{align*}
\end{lemma}

\begin{theorem}\label{thm:BiContNet-gen-fin}
The operator $(A_C,\dom(A_C))$, defined in \eqref{eqn:Gen}, generates a bi-con\-ti\-nu\-ous semigroup $(T_C(t))_{t\geq0}$ on $X=\L^{\infty}\left(\left[0,1\right],\CC^m\right)$.
\end{theorem}

\begin{proof}
We first show that {the} operator $A_C$ is bi-densely defined. Take any $f\in\L^{\infty}\left(\left[0,1\right],\CC^m\right)$. For $n\in\NN$ let $\Omega_n:=\left[\frac{1}{n},1-\frac{1}{n}\right]\subseteq\left[0,1\right]$ and define $f_n:\left[0,1\right]\to\CC^m$ by a linear truncation of $f$ outside $\Omega_n$, i.e.,
\[
f_n(x):=\begin{cases}
nf\left(\frac{1}{n}\right)x,&\ x\in\left[0,\frac{1}{n}\right],\\
f(x),&\ x\in\left[\frac{1}{n},1-\frac{1}{n}\right],\\
nf\left(1-\frac{1}{n}\right)(1-x),&\ x\in\left[1-\frac{1}{n},1\right].
\end{cases}
\]
Observe that  $f_n$ is Lipschitz for each $n\in\NN$ and hence $f_n\in\W^{1,\infty}\left(\left[0,1\right],\CC^m\right)$. 
Moreover $f_n(1)=f_n(0)=0$ for each $n\in\NN$ implying that ${f_n(1)}=\BB^Cf_n(0)$, hence $f_n\in\dom(A_C)$.
Furthermore, one has that $\sup_{n\in\NN}{\left\|f_n\right\|}\le \|f\|<\infty$ and  $f_n{\rightarrow}f$ as $n\to\infty$ with respect to the weak$^*$-topology since
\begin{align*}
{\left|\int_0^1 {\left\langle \left(f_n(x)-f(x)\right),g(x)\right\rangle\ \dd{x}}\right|} &\leq2\left\|f\right\|_{\infty}\lambda^1\left(\left[0,\frac{1}{n}\right]\cup\left[1-\frac{1}{n},1\right]\right)\left\|g\right\|_{1} \\
&=\frac{4}{n}\left\|f\right\|_{\infty}\left\|g\right\|_1
\end{align*}
for each $g\in\L^1\left(\left[0,1\right],\CC^m\right)$. 

We now define a sequence of operators $A_n$ approximating $A_C$ in the following way. 
 For each $c_j\in\RR$ there exists a sequence $\left(c^{(n)}_j\right)_{n\in\NN}$ in $\QQ$ such that $\lim_{n\rightarrow\infty}{c_j^{(n)}}=c_j$. 
Since the network is finite, for each $n\in\NN$ the velocities $c^{(n)}_j$, $j\in J$, satisfy condition \eqref{pj} and, by Proposition \ref{thm:BiContNet-Q} we obtain {a} contraction bi-continuous semigroup $(T_n(t))_{t\geq0}$ generated by
\begin{equation}
\begin{aligned}\label{An}
A_n &:= \diag\left(c_j^{(n)}\cdot\frac{\dd}{\dd x}\right),\\
\dom(A_n) &:= \left\{f\in\W^{1,\infty}\left(\left[0,1\right],\CC^m\right)\mid  f(1)=\BB^{C_n}f(0) \right\},
\end{aligned}
\end{equation}
where $C_n:=\diag\left(c_j^{(n)}\right)$. Moreover, all semigroups  $(T_n(t))_{t\geq0}$, $n\in \NN$,  are similar and thus uniformly bi-continuous of type 0.

Observe, that the general assumptions of Theorem \ref{thm:TK} are satisfied. Let us now check the assumptions of assertion (b).  
 Let $R:=R(\lambda,A_C)$ and observe that $R:\L^{\infty}\left(\left[0,1\right],\CC^m\right)\to\dom(A_C)$ is a bijection. By above, $\mathrm{Ran}(R)$ is bi-dense  in $\L^{\infty}\left(\left[0,1\right],\CC^m\right)$.
 For every $n\in\NN$, replacing   $c_j$ by $c_j^{(n)}$ for all $j\in J$,  Lemma \eqref{prop:ExRes} yields an explicit expression for $R(\lambda,A_n)$. It is easy to see that  $R(\lambda,A_n)f\stackrel{\left\|\cdot\right\|}{\rightarrow} Rf$ for $f \in D(A_C)$ as $n\to\infty$. Applying Theorem \ref{thm:TK} gives us a bi-continuous semigroup $(T_C(t))_{t\geq0}$ with generator $(B,\dom(B))$. Note that, since in our case $R=R(\lambda,A_C)$ is a resolvent, by  Remark \ref{rem:TK} we have $R=R(\lambda,A_C)=R(\lambda,B)$ for $\lambda\in\rho(A_C)$ and by the uniqueness of the Laplace transform we conclude that $(B,\dom(B))=(A_C,\dom(A_C))$. 
\end{proof}

\begin{corollary}\label{cor:WellPGen}
The flow problem \eqref{eqn:F} is well-posed on $X=\L^{\infty}\left(\left[0,1\right],\CC^m\right)$.
\end{corollary}

\begin{remark}
In the same manner, by using the original strongly continuous version of the Trotter--Kato Theorem (see \cite[Sect.~III.4b]{EN}), one can deduce the well-posedness of the problem on $X=\L^1 \left(\left[0,1\right],\CC^m\right)$. 
\end{remark}

{
\begin{remark}
Note, that by our assumptions $\|\mathbb{B}^{C}_{\lambda}\|<1$ (see also \cite[(18.20)]{Positive2017}), hence von-Neuman expansion yields positivity of the resolvent $R(\lambda,A_C)$ given in Lemma \ref{prop:ExRes} (which can be proven also in the case of an infinite graph). By \cite[Them.~1.4.1]{FaPHD}, 
$(A_C,\dom(A))$ thus generates a positive bi-continuous semigroup.
Further, some spectral properties can be deduced from the resolvent formula as in \cite{Dorn2008}.
\end{remark}

\subsection*{Conclusion}
We have shown that the abstract theory of bi-continuous operator semigroups can be successfully applied to study transport processes in infinite networks. This opens the door for various possible further considerations. In fact, our setting is already used in  \cite{Dob2020} where the long-term behaviour of the solutions in the spirit of \cite{DKS2009} is studied. 
{By using a concept of the semigroup at infinity and recently developed asymptotical theory,} it is shown that, under suitable assumptions on the graph, the bi-continuous semigroups obtained  in  Theorem \ref{thm:BiContNet} {and} and  Theorem  \ref{thm:BiContNet-Q} {both} behave asymptotically periodic with respect to the operator norm. {In the case of general velocities, the question of well-posedness  of problem \eqref{eqn:F} as well as the description of the long-time behaviour of its solutions in $\L^{\infty}\left(\left[0,1\right],\ell^1\right)$  remain open.}
}



\begin{thebibliography}{10}

\bibitem{AK2002}
A.~Albanese and F.~K\"{u}hnemund.
\newblock Trotter-{K}ato approximation theorems for locally equicontinuous
  semigroups.
\newblock {\em Riv. Mat. Univ. Parma (7)}, 1:19--53, 2002.

\bibitem{ALM}
A.~A. Albanese, L.~Lorenzi, and V.~Manco.
\newblock Mean ergodic theorems for bi-continuous semigroups.
\newblock {\em Semigroup Forum}, 82(1):141--171, Feb 2011.

\bibitem{AM2004}
A.~A. Albanese and E.~Mangino.
\newblock {T}rotter-{K}ato theorems for bi-continuous semigroups and
  applications to {F}eller semigroups.
\newblock {\em Journal of Mathematical Analysis and Applications}, 289(2):477
  -- 492, 2004.

\bibitem{NagelPos1986}
W.~Arendt, A.~Grabosch, G.~Greiner, U.~Groh, H.~P. Lotz, U.~Moustakas,
  R.~Nagel, F.~Neubrander, and U.~Schlotterbeck.
\newblock {\em One-parameter semigroups of positive operators}, volume 1184 of
  {\em Lecture Notes in Mathematics}.
\newblock Springer-Verlag, Berlin, 1986.

\bibitem{BanFal2015}
J.~{Banasiak} and A.~{Falkiewicz}.
\newblock {Some transport and diffusion processes on networks and their graph
  realizability.}
\newblock {\em {Appl. Math. Lett.}}, 45:25--30, 2015.

\bibitem{BanFal2017}
J.~{Banasiak} and A.~{Falkiewicz}.
\newblock {A singular limit for an age structured mutation problem.}
\newblock {\em {Math. Biosci. Eng.}}, 14(1):17--30, 2017.

\bibitem{BaNam2014}
J.~Banasiak and P.~Namayanja.
\newblock Asymptotic behaviour of flows on reducible networks.
\newblock {\em Netw. Heterog. Media}, 9(2):197--216, 2014.

\bibitem{BanPuc2018}
J.~{Banasiak} and A.~{Puchalska}.
\newblock {Generalized network transport and Euler-Hille formula.}
\newblock {\em {Discrete Contin. Dyn. Syst., Ser. B}}, 23(5):1873--1893, 2018.

\bibitem{Positive2017}
A.~B{\'a}tkai, M.~Fijav{\v{z}}, and A.~Rhandi.
\newblock {\em Positive Operator Semigroups: From Finite to Infinite
  Dimensions}.
\newblock Operator Theory: Advances and Applications. Springer International
  Publishing, 2017.

\bibitem{BDK2013}
F.~{Bayazit}, B.~{Dorn}, and M.~K. {Fijav\v{z}}.
\newblock {Asymptotic periodicity of flows in time-depending networks.}
\newblock {\em {Netw. Heterog. Media}}, 8(4):843--855, 2013.

\bibitem{BF}
C.~Budde and B.~Farkas.
\newblock Intermediate and extrapolated spaces for bi-continuous operator
  semigroups.
\newblock {\em J. Evol. Equ.}, 19(2):321--359, 2019.

\bibitem{DiestelUhl1977}
J.~Diestel and J.~Uhl.
\newblock {\em Vector Measures}.
\newblock Mathematical surveys and monographs. American Mathematical Society,
  1977.

\bibitem{DiPLio1989}
R.~J. DiPerna and P.-L. Lions.
\newblock Ordinary differential equations, transport theory and {S}obolev
  spaces.
\newblock {\em Invent. Math.}, 98(3):511--547, 1989.

\bibitem{Dob2020}
A.~Dobrick.
\newblock On the asymptotic behaviour of semigroups for flows in infinite
  networks.
\newblock arXiv:2011.07014, 2020.

\bibitem{Dorn2008}
B.~Dorn.
\newblock Semigroups for flows in infinite networks.
\newblock {\em Semigroup Forum}, 76(2):341--356, Mar 2008.

\bibitem{DKNR2010}
B.~Dorn, M.~K. Fijav{\v{z}}, R.~Nagel, and A.~Radl.
\newblock The semigroup approach to transport processes in networks.
\newblock {\em Physica D: Nonlinear Phenomena}, 239(15):1416 -- 1421, 2010.
\newblock Evolution Equations in Pure and Applied Sciences.

\bibitem{DKS2009}
B.~{Dorn}, V.~{Keicher}, and E.~{Sikolya}.
\newblock {Asymptotic periodicity of recurrent flows in infinite networks.}
\newblock {\em {Math. Z.}}, 263(1):69--87, 2009.

\bibitem{Vertex2008}
K.-J. Engel, M.~K. Fijav\v{z}, R.~Nagel, and E.~Sikolya.
\newblock Vertex control of flows in networks.
\newblock {\em Networks \& Heterogeneous Media}, 3:709, 2008.

\bibitem{EN}
K.-J. Engel and R.~Nagel.
\newblock {\em One-parameter semigroups for linear evolution equations}, volume
  194 of {\em Graduate Texts in Mathematics}.
\newblock Springer-Verlag, New York, 2000.

\bibitem{FaPHD}
B.~Farkas.
\newblock {\em Perturbations of Bi-Continuous Semigroups}.
\newblock PhD thesis, E\"otv\"os Lor\'and University, 2003.

\bibitem{FaStud}
B.~Farkas.
\newblock Perturbations of bi-continuous semigroups.
\newblock {\em Studia Math.}, 161(2):147--161, 2004.

\bibitem{FaSF}
B.~Farkas.
\newblock Perturbations of bi-continuous semigroups with applications to
  transition semigroups on {$C_b(H)$}.
\newblock {\em Semigroup Forum}, 68(1):87--107, 2004.

\bibitem{Farkas2011}
B.~Farkas.
\newblock Adjoint bi-continuous semigroups and semigroups on the space of
  measures.
\newblock {\em Czechoslovak Mathematical Journal}, 61(2):309--322, Jun 2011.

\bibitem{KS2005}
M.~Kramar and E.~Sikolya.
\newblock Spectral properties and asymptotic periodicity of flows in networks.
\newblock {\em Mathematische Zeitschrift}, 249(1):139--162, Jan 2005.

\bibitem{KuPHD}
F.~K\"uhnemund.
\newblock {\em {B}i-Continuous Semigroups on Spaces with Two Topologies:
  {T}heory and Applications}.
\newblock PhD thesis, {E}berhard-{K}arls-{U}niversit\"at T\"ubingen, 2001.

\bibitem{Ku}
F.~K\"uhnemund.
\newblock A {H}ille-{Y}osida theorem for bi-continuous semigroups.
\newblock {\em Semigroup Forum}, 67(2):205--225, 2003.

\bibitem{Lotz1985}
H.~P. Lotz.
\newblock Uniform convergence of operators on {$L^\infty$} and similar spaces.
\newblock {\em Math. Z.}, 190(2):207--220, 1985.

\bibitem{MS2007}
T.~M\'{a}trai and E.~Sikolya.
\newblock Asymptotic behavior of flows in networks.
\newblock {\em Forum Math.}, 19(3):429--461, 2007.

\bibitem{Scirrat2016}
A.~K. Scirrat.
\newblock {\em Evolution Semigroups for Well-Posed, NonAutonomous Evolution
  Families}.
\newblock PhD thesis, Louisiana State University and Agricultural and
  Mechanical College, 2016.

\bibitem{vanZuilen}
W.~van Zuijlen.
\newblock Integration of functions with values in a {R}iesz space.
\newblock Master's thesis, {R}adboud {U}niversiteit {N}ijmegen, 2012.

\bibitem{vBeLu2005}
J.~von Below and J.~A. Lubary.
\newblock The eigenvalues of the {L}aplacian on locally finite networks.
\newblock {\em Results Math.}, 47(3-4):199--225, 2005.

\bibitem{vBeLu2009}
J.~von Below and J.~A. Lubary.
\newblock The eigenvalues of the {L}aplacian on locally finite networks under
  generalized node transition.
\newblock {\em Results Math.}, 54(1-2):15--39, 2009.

\end{thebibliography}

\end{document}